\documentclass[11pt]{article}
\usepackage{mathrsfs,amssymb,amsmath,amsthm}

\makeatletter
	
	\@addtoreset{equation}{section}
\makeatother

\begin{document}
\renewcommand{\theenumi}{\rm (\roman{enumi})}
\renewcommand{\labelenumi}{\rm \theenumi}
\newtheorem{Th}{Theorem}[section]
\newtheorem{df}[Th]{Definition}
\newtheorem{lm}[Th]{Lemma}
\newtheorem{pr}[Th]{Proposition}
\newtheorem{co}[Th]{Corollary}
\newtheorem{ex}[Th]{Example}
\allowdisplaybreaks

\title{Recurrence of the Brownian motion in multidimensional semi-selfsimilar environments and Gaussian environments}
\author {Seiichiro Kusuoka$^1$\footnote{e-mail: kusuoka@math.tohoku.ac.jp}, Hiroshi Takahashi$^2$ and Yozo Tamura$^3$
\vspace*{0.1in}\\
$ ^{1}$Graduate School of Science, Tohoku University \\
6-3 Aramaki Aza-Aoba, Aoba-ku, Sendai 980-8578 Japan. \vspace*{0.1in} \\
$ ^{2}$College of Science and Technology, Nihon University\\
7-24-1 Narashinodai, Funabashi 274-8501 Japan. \vspace*{0.1in} \\
$ ^3$Faculty of Science and Technology, Keio University\\
3-14-1 Hiyoshi, Kohoku-ku, Yokohama 223-8522, Japan.}

\maketitle

\begin{abstract}
Asymptotic behavior of the one-dimensional Brownian motion in general random environments has been investigated by many researchers. However, many of the methods used in the argument are available only for the one-dimensional case. In this paper the multi-dimensional case of the problem is considered, and we obtain some sufficient conditions for recurrence of the multi-dimensional Brownian motion in random environments. By using the sufficient conditions we show that the recurrence of the Brownian motion in Gaussian environments under some conditions on the correlation functions.
\end{abstract}

{\bf 2010 AMS Classification Numbers:} 60K37, 60J60, 60G60.

 \vskip0.2cm


\section{Introduction}

It is well-known that the $d$-dimensional Brownian motion, 
is recurrent if $d=1$ or $2$, and transient if $d\geq 3$.
More generally, we have many criterions for the recurrence and transience of diffusion processes, and whether diffusion processes are recurrent or transient depends on their generators (see \cite{Ichihara}, Section 3 and 4 of Chapter IV in \cite{Ikeda-Watanabe}, Section 1.6 in \cite{FOT}, etc.).
In this paper, we consider sufficient conditions for the recurrence of the Brownian motion in random environments.

The Brownian motion in random environments is formulated as follows.
Let $W$ be a random Borel measurable function on ${\bf R}^d$ and consider the following generator
\begin{equation}\label{generator0}
\frac 12 \left( \triangle - \nabla W \cdot \nabla \right) .
\end{equation}
Symbolically, the diffusion process $X_W$ associated with this generator is the solution to the stochastic differential equation: 
\begin{equation}\label{SDE0}
dX_W(t) = dB(t) - \frac 12 (\nabla W)(X_W(t)) dt
\end{equation}
where $(B(t))$ is the $d$-dimensional Brownian motion.
We call $W$ an environment and $X_W$ the Brownian motion in $W$.
When $W$ is not differentiable, (\ref{SDE0}) is no more than symbolical.
However, since
\[
\frac 12 \left( \triangle - \nabla W \cdot \nabla \right) = \frac 12 e^W \sum _{k=1}^d \frac{\partial}{\partial x_k} \left( e^{-W} \frac{\partial}{\partial x_k}\right),
\]
we can construct the diffusion process $X_W$ associated with (\ref{generator0}) by a random time-change of the diffusion process associated with the Dirichlet form:
\begin{equation}\label{DF0}
{\mathscr E}(f,g) = \frac 12 \int _{{\bf R}^d} \left( \nabla f \cdot \nabla g \right) e^{-W}dx.
\end{equation}
Hence, the existence of the diffusion process $X_W$ associated with (\ref{generator0}) under suitable conditions on $W$ is guaranteed (see \cite{FOT}).
Our interest is in the recurrence of the diffusion process $X_W$ given by (\ref{generator0}).

The topic of stochastic processes in random environments was originally given by Solomon \cite{Solomon}.
He introduced the one-dimensional random walk in a random environment and obtained some asymptotic behavior.
In particular, a sufficient condition for the recurrence was obtained.
Later, Sinai \cite{Sinai} obtained another property of asymptotic behavior (so-called localization) of the one-dimensional random walk in a random environment.
Brox \cite{Brox} formulated the one-dimensional Brownian motion in a random environment, which is an analogue to a continuous version of Sinai's model, and obtained the same asymptotic behavior as Sinai's result.
After that, Kawazu, Tamura and Tanaka \cite{ktt} extended Brox's result to general environments.

We are also interested in asymptotic behaviors of the multi-dimensional Brownian motion in random environments.
The approaches in \cite{Brox} and \cite{ktt} are special for one-dimensional diffusion processes and not available for the multi-dimensional ones (see e.g. \cite{ktt}).
In the multi-dimensional case there are a few results.
Tanaka \cite{Tanaka} considered the multi-dimensional Brownian motion in the environment which is L\'evy's multi-dimensional Brownian motion, and showed the recurrence of the process.
The case of the multi-dimensional Brownian motion in the environment generated by independent reflected Brownian motions is concerned in \cite{Taka}.
Some asymptotic behavior different from the recurrence of the multi-dimensional Brownian motion in some environments are obtained in \cite{Kim} and \cite{mathieu95}.

In this paper, we consider more general environments and extend the results of multi-dimensional cases.
First, we consider some sufficient conditions for the recurrence of the Brownian motion in a general environment $W$.
We use the generalized version of Ichihara's criterion.
Ichihara's criterion enables us to obtain the recurrence by checking conditions on coefficients of generators (see \cite{Ichihara}).
Since we are interested in the case that $W$ is not smooth, we apply the generalized version of  Ichihara's criterion given by the argument of Dirichlet forms (Theorem 1.6.7 in \cite{FOT}).
By using it, we obtain some criterions for the recurrence of the Brownian motion in general environments.

Second, we consider the Brownian motion in Gaussian environments and some sufficient conditions for the recurrence of the Brownian motion.
Since Gaussian fields are well-investigated, it is possible to give clear sufficient conditions for the recurrence.
We show that two conditions on correlation functions imply the recurrence.
One of them is for the property that uniformly positive environments on the disk-type region appear with positive probability.
To obtain the positivity of the probability, we extend Tanaka's argument in \cite{Tanaka} to abstract Wiener spaces.
For the extension, we study the relation between the behavior of Gaussian fields and Cameron-Martin spaces.
The other one is for the ergodicity of the environment.
This part goes with the same argument in \cite{Tanaka}.
This criterion is applicable to the Brownian motion on fractional Brownian fields.
We give it as an example of an application of our result.

This paper is constructed as follows.
In Section \ref{General} some abstract sufficient conditions for the recurrence of the Brownian motion in a general environment $W$, and in Section \ref{Gauss} the case of Gaussian environments are considered.
In Section \ref{sec:fBf} we consider the case that the environments are fractional Brownian fields. Applying the criterion obtained in Section \ref{Gauss}, we show the recurrence of the Brownian motion on fractional Brownian fields.

\section{Recurrence of the Brownian motion in general environments}\label{General}

First we consider the Brownian motion in a non-random environment.
Later, we will randomize environments.
Let $W$ be a locally bounded and measurable function on ${\bf{R}}^d$, and let $W(0)=0$.
Consider the diffusion process $X_W$ with initial value $X(0)$ given by the generator:
\begin{equation}\label{generator}
\frac 12 \left( \Delta -\nabla W \cdot \nabla \right) = \frac 12 e^W \sum _{k=1}^d \frac{\partial}{\partial x_k} \left( e^{-W} \frac{\partial}{\partial x_k}\right) .
\end{equation}
As mentioned in Introduction, (\ref{generator}) is a symbolic representation when $W$ is not differentiable.
To define $X_W$, we consider another diffusion process $Y_W$ with initial value $X(0)$ generated by the Dirichlet form
\[
{\mathscr E}(f,g) = \frac 12 \int _{{\bf R}^d} \left( \nabla f \cdot \nabla g \right) e^{-W}dx, \quad f,g \in {\mathcal D}
\]
where
\[
{\mathcal D} := \left\{ f\in L^2(e^{-W}dx); \int  _{{\bf R}^d} |\nabla f|^2 e^{-W} dx < \infty \right\} .
\]
For the existence and the uniqueness of $Y_W$, see Chapter 7 in \cite{FOT}.
The generator of $Y_W$ is symbolically given by 
\begin{equation}\label{generator2}
\frac 12 \sum_{k=1}^{d} \frac{\partial}{\partial x_k} \left( e^{-W} \frac{\partial}{\partial x_k}\right). 
\end{equation}
We set 
\[
\tau_t:= \int_0^t \exp\left\{ -W(Y_W(s)) \right\} ds \quad \mbox{and} \quad A_t := \tau_t^{-1}.
\]
We construct $X_W$ by the time changed process
\[
X_W(t) := Y_W(A_t).
\]
Then, the generator of $X_W$ is symbolically given by (\ref{generator}).
We remark that; when $W$ is Lipschitz continuous, the representation of the generators above will be not only symbolic, but also  rigorous.
Thus, we define the diffusion process $X_W$ associated to (\ref{generator}).
Since $\tau_t$ is strictly increasing and $\lim _{t\rightarrow \infty} \tau _t = \infty$ almost surely, the recurrent and transient property of $X_W$ coincides with that of $Y_W$.
Hence, we study the behavior of the process $Y_W$ generated by (\ref{generator2}) instead of $X_W$ generated by (\ref{generator}).
We remark that the recurrent and transient property of $Y_W$ is independent of initial value $X(0)$, because of the Markov property of $Y_W$. 

Fix $d\in {\bf N}\setminus \{ 1\}$ and $r>1$. Let $E_n$ be the set \( \{ x\in {\bf{R}}^d : |x|<r^n \}\) for $n\in {\bf Z}$, or \( \{ x\in {\bf{R}}^d : |x_i|<r^n,\ \mbox{for } i=1,2,\dots ,d \}\) for $n\in {\bf Z}$. We denote \( E_n \setminus E_{n-1}\) by $D_n$ for $n\in {\bf Z}$.

\begin{lm}\label{lm1-1}
If there exists a constant $c \in {\bf{R}}$ such that
\begin{equation}
\inf _{x\in D_n}W(x) \geq n(d-2)\log r -c \label{1-0}
\end{equation}
holds for infinitely many $n \in {\bf{N}}$, then $X_W$ is recurrent.  
\end{lm}

\begin{proof}
Since \( \{ x\in {\bf{R}}^d : |x|<r^n \}\) and \( \{ x\in {\bf{R}}^d : |x_i|<r^n,\ \mbox{for } i=1,2,\dots ,d \}\) are homeomorphic, it is sufficient to consider only the case that $E_n$ is the set \( \{ x\in {\bf{R}}^d : |x|<r^n \}\).
As mentioned above, we check the recurrence of $Y_W$ instead of $X_W$.
By Theorem 1.6.7 in \cite{FOT}, it is sufficient to show that
\begin{equation}
\int _1^{\infty} \left\{ \int _{S^{d-1}} e^{-W(s\theta )} d\theta \right\} ^{-1} s^{-d+1}ds = \infty \label{1-1}
\end{equation}
almost surely where $S^{d-1}$ is the surface of the $d$-dimensional unit ball centered at the origin and $d\theta$ is the volume measure on $S^{d-1}$.
Let $A$ be the set of $n\in {\bf{N}}$ which satisfies (\ref{1-0}), $\pi _d$ the surface area of $S^{d-1}$. Let $d \not = 2$. Since the case that $d = 2$ can be discuss similarly, we omit the proof of the case.
By changing variables we have
\begin{eqnarray*}
\left( \mbox{the left hand side of (\ref{1-1})}\right)
&=& \sum _{n=0}^{\infty} \int _1^r \left\{ \int _{S^{d-1}} e^{-W(r^ns\theta )} d\theta \right\} ^{-1} (r^ns)^{-d+1}r^n ds \\
&\geq& \pi _d ^{-1} \sum _{n\in A} \int _1^r e^{n(d-2)\log r -c} (r^ns)^{-d+1}r^n ds\\
&=& \frac {\pi _d ^{-1} e^{-c}}{-d+2} \sum _{n\in A} \left( r^{-d+2}-1 \right) .
\end{eqnarray*}
The assumption implies the last term is infinite.
\end{proof}

Next, we randomize the environments.
Let $(\Omega, {\mathscr F}, P)$ be a probability space and $W$ a $C({\bf R}^d)$-valued Borel measurable random variable where the topology of $C({\bf R}^d)$ is equipped with locally uniform convergence such that $W(0)=0$ almost surely.
Let ${\mathscr B}(C({\bf R}^d))$ be the total family of the Borel measurable sets in $C({\bf R}^d)$.
For each $W$ we define the Markov system $\{(X_W(t)), P_W^x, x\in {\bf{R}}^d \}$ given by the generator (\ref{generator}).

We assume that there exist $r>1$ and $\alpha >0$ such that the law of $TW$ equals to that of $W$ where $T$ is a mapping from Borel measurable functions on ${\bf{R}}^d$ to themselves defined by \( T f(x):= r^{-\alpha }f(rx)\).
Then, $T$ is a measure preserving transformation.
We call such a $W$ with a measure preserving transformation $T$ a semi-selfsimilar random environment.
Define $D_n$ as above, and let $T_n := T^n$ for $n\in {\bf Z}$.

\begin{Th}\label{Th1-2}
If $T$ is weakly mixing, i.e. 
\[
\lim _{n\rightarrow \infty} \frac 1n \sum _{k=0}^{n-1}\left| P(  W \in  (T_k A) \cap B)- P( W \in A)P( W \in B)\right| =0,\quad A,B \in {\mathscr B}(C({\bf R}^d)),
\]
and if there exists a positive constant $a$ such that
\[
P \left( \inf _{x\in D_1}W(x) \geq a \right) >0,
\]
then $X_W$ is recurrent for almost all environments $W$.
\end{Th}

\begin{proof}
It is sufficient to show that $W$ satisfies (\ref{1-0}) almost surely.
Since $T$ is weakly mixing, so is ergodic (see Theorem 1.17 in \cite{Walters}). By Theorem 1.5 in \cite{Walters} we have
\[
P\left( \bigcup _{n=0} ^\infty \{ T_n W\in A\}\right) =1 \quad \mbox{for}\ A \in {\mathscr B}(C({\bf R}^d))\ \mbox{such that}\ P(W \in A)>0.
\]
Thus, by choosing $A$ as $\{ f \in C({\bf R}^d); \inf _{x\in D_1} f \geq a\}$, we have
\[
P\left( \bigcap _{N=0}^\infty \bigcup _{n=N} ^\infty \{ \inf _{x\in D_1}T_n W(x) \geq a\} \right) =1.
\]
Hence, $\{ \inf _{x\in D_1}T_n W(x) \geq a \}$ occurs infinitely many times almost surely.
This means that \( \{ \inf _{x\in D_n} W(x) \geq ar^{\alpha n} \}\) occurs infinitely many times almost surely.
Therefore, $W$ satisfies (\ref{1-0}) almost surely.
\end{proof}

Next we consider a more specific case: the case of selfsimilar random environments.
We assume that there exists $\alpha >0$ such that the law of $T_tW$ equals to that of $W$ for all $t\in {\bf{R}}$ where $T_t$ is a mapping from Borel measurable functions on ${\bf{R}}^d$ to themselves defined by \( T_t f(x):= 2^{-\alpha t}f(2^tx)\) for $t\in {\bf{R}}$.
Then $\{ T_t; t\in {\bf{R}}\}$ is a one-parameter family of measure preserving transformations.
We call such a $W$ with a one-parameter family of measure preserving transformations $\{ T_t; t\in {\bf{R}}\}$ a selfsimilar random environment.

\begin{Th}\label{Th1-3}
If $\{ T_t; t\in {\bf{R}}\}$ is weakly mixing, i.e.
\[
\lim _{n\rightarrow \infty} \frac 1t \int _0^t \left| P(  W \in  (T_t A) \cap B)- P( W \in A)P( W \in B)\right| =0,\quad A,B \in {\mathscr B}(C({\bf R}^d)),
\]
and if there exists a positive constant $a$ such that
\begin{equation}\label{eq:Th1-3-10}
P \left( \inf _{x\in S^{d-1}}W(x) \geq a \right) >0,
\end{equation}
then $X_W$ is recurrent for almost all environments $W$.
\end{Th}

\begin{proof}
By Theorem 1.6.7 in \cite{FOT}, it is sufficient to show that
\begin{equation}
\int _1^{\infty} \left\{ \int _{S^{d-1}} e^{-W(s\theta )} d\theta \right\} ^{-1} s^{-d+1}ds = \infty \label{Th1-3-1}
\end{equation}
almost surely.
Let $M(t):=  \inf _{x\in S^{d-1}}(T_tW)(x)$.
Then,
\begin{align}
\int _1^\infty \left\{ \int _{S^{d-1}}e^{-W(s\theta )} d\theta \right\} ^{-1} s^{-d+1} ds
&= (\log 2) \int _0^\infty 2^{(2-d)s} \left\{ \int _{S^{d-1}}e^{-2^{\alpha s}(T_s W)(\theta )} d\theta \right\} ^{-1} ds \nonumber \\
&\geq \pi _d^{-1}(\log 2) \int _0^\infty 2^{(2-d)s} \exp \{ 2^{\alpha s} M(s)\} ds \nonumber \\
&\geq \pi _d^{-1}(\log 2) \int _{t_0} ^\infty {\bf I}_{(a,\infty )} (M(s))ds \label{Th1-3-2}
\end{align}
where $t_0:= \inf \{ t>0;\ 2^{\alpha t} a+ (2-d)t(\log 2) >0 \}$.
It holds that $t_0 <\infty$, since $a>0$ and $\alpha >0$.
The fact that weakly mixing yields ergodicity is also true for one-parameter families of measure preserving transformations (see Section 8 of in \cite{Arnold} and the remark on pages 22-23 in \cite{Halmos}).
Hence, we have
\begin{align*}
\lim _{t\rightarrow \infty }\frac 1t \int _0^t {\bf I}_{(a,\infty )} (M(s))ds
&= E\left[ {\bf I}_{(a,\infty )} (M(0))\right] \\
&= P \left( \inf _{x\in S^{d-1}}W(x) \geq a \right) .
\end{align*}
This equality and (\ref{eq:Th1-3-10}) imply
\[
\int _0^\infty {\bf I}_{(a,\infty )} (M(s))ds = \infty .
\]
Therefore, (\ref{Th1-3-2}) yields (\ref{Th1-3-1}).
\end{proof}

\section{Recurrence of the Brownian motion in Gaussian environments}\label{Gauss}

When $W$ is a Gaussian field, the sufficient condition for the recurrence can be simplified.
Let $W$ be a Gaussian field on ${\bf{R}}^d$ i.e. $(W(x); x\in {\bf R}^d)$ is a family of random variables such that the ${\bf{ R}}^n$-valued random variable $(W(x_1), W(x_2),\dots , W(x_n))$ has an $n$-dimensional Gaussian distribution for all $n\in N$ and $x_1,x_2,\dots ,x_n \in {\bf R}^d$. We assume that $W$ is continuous on ${\bf{R}}^d$ almost surely, $W(0)=0$ almost surely, and that $E[W(x)]=0$ for $x\in {\bf{R}}^d$.
Let $K(x,y):=E[W(x)W(y)]$ for $x,y \in {\bf R}^d$.
Define $\{ D_n ; n \in {\bf Z}\}$ as in Section \ref{General}.

First, we prepare the following lemma.

\begin{lm}\label{lm2-1}
If there exists a positive constant $\varepsilon$ such that
\[
\inf _{x\in D_1}\int _{D_1} K(x,y)dy \geq \varepsilon ,
\]
then for all $a \in {\bf{R}}$
\[
P\left( \inf _{x\in D_1}W(x)\geq a \right) >0.
\]
\end{lm}

\begin{proof}
It is sufficient to show the assertion for $a>0$. Define a random function $\tilde W$ on $D_1$ by $\tilde W(x) :=W(x)$ for $x\in D_1$, which is the restriction of $W$ on $D_1$.
Then, $\tilde W$ can be regarded as an $L^2(D_1,dx)$-valued random variable, where $dx$ is the $d$-dimensional Lebesgue measure. Define a linear operator $A$ from $L^2(D_1,dx)$ to $L^2(D_1,dx)$ by
\[
Af(x) := \int _{D_1} K(x,y)f(y)dy, \quad x\in {\bf R}^d.
\]
It is known that $A$ is bounded and non-negative definite, and it holds that
\[
\int e^{i\langle f,W\rangle}dP = e^{-\frac 12\langle Af,f\rangle}, \quad f \in L^2(D_1,dx).
\]
Hence, the image of the linear operator $\sqrt A$ is the Cameron-Martin space $H$ of the law of $\tilde W$.
We remark that \( H \subset C_b(D_1) \), because $\tilde W$ is continuous almost surely.
For $f \in H$ and $\delta >0$ we define
\[
B(f,\delta ):=\left\{ g\in C_b(D_1);\ \sup _{x\in D_1}|f(x)-g(x)|<\delta \right\} .
\]
Fix $f \in H$ and $\delta >0$.
Now we assume that \( P( \tilde W \in B(f,\delta ))=0\), and we will obtain a contradiction.
Since $P\circ \tilde W^{-1}$ is absolutely continuous with respect to that with shifts for the direction $h \in H$ (see e.g. Theorem 1.3 in \cite{Sh}), we have
\begin{equation}\label{eq:10}
P\left( \tilde W \in \bigcup _{g\in H} B(g,\delta ) \right) =0.
\end{equation}
On the other hand, in view of Theorem 3.6.1 in \cite{Bogachev}, the completion $\overline{H}^{\| \cdot \| _{C_b(D_1)}}$ of $H$ with respect to $\| \cdot \| _{C_b(D_1)}$ coincides with the topological support of $P\circ \tilde W^{-1}$, i.e.
\[
P\left( \tilde W\in \overline{H}^{\| \cdot \| _{C_b(D_1)}} \right) =1.
\]
This contradicts (\ref{eq:10}).
Thus, we have
\begin{equation}\label{lm2-1-1}
P( \tilde W \in B(f,\delta ))>0\quad \mbox{for all}\ f \in H\ \mbox{and}\ \delta >0.
\end{equation}
Since $1\in L^2(D_1,dx)$, we have \( A1 \in H\). By the assumption we have
\[
(A1)(x)=\int _{D_1} K(x,y)dy \geq \varepsilon , \quad x\in D_1.
\]
Hence, by (\ref{lm2-1-1}) we obtain
\[
P\left( \inf _{x\in D_1}W(x)\geq a \right) \geq P\left( W \in B\left( {\frac {a+1}\varepsilon }A1,1\right) \right) >0.
\]
\end{proof}

By a similar argument to the proof of Lemma \ref{lm2-1}, we have the following corollary.

\begin{co}\label{co2-1}
If there exists a positive constant $\varepsilon$ such that
\[
\inf _{x\in S^{d-1}}\int _{S^{d-1}} K(x,y)dy \geq \varepsilon ,
\]
then for all $a \in {\bf{R}}$
\[
P\left( \inf _{x\in S^{d-1}}W(x)\geq a \right) >0.
\]
\end{co}

Let $r>1$ and $\alpha >0$, and define a mapping $T$ from Borel measurable functions on ${\bf{R}}^d$ to themselves by \( T f(x):= r^{-\alpha }f(rx)\).
We define a transformation $T_n$ as in Section \ref{General}.
Now we consider the sufficient condition for $T_n$ to be strongly mixing with respect to the measures associated with semi-selfsimilar Gaussian environments.

\begin{lm}\label{lm2-2}
Suppose that the law of $T_n W$ equals to that of $W$ for all $n$ and that
\[
\lim _{n\rightarrow \infty} r^{-\alpha n}\sup _{x,y\in D_1} K(r^n x,y)=0,
\]
then the family of the transformations $\{ T_n\}$ is strongly mixing.
\end{lm}

\begin{proof}
In view of Theorem 1.17 in \cite{Walters} it is sufficient to prove that for \( x_1,\dots ,x_N, y_1, \dots ,y_M \in {\bf{R}}^d\), \( f \in C_b({\bf{R}}^N)\), and \( g \in C_b({\bf{R}}^M)\),
\begin{eqnarray*}
&&\lim _{n\rightarrow \infty} \int f(W(x_1),\dots ,W(x_N))g(T^n W(y_1),\dots ,T^n W(y_M))dP\\
&&= \int f(W(x_1),\dots ,W(x_N))dP \int g(W(y_1),\dots ,W(y_M))dP .
\end{eqnarray*}
We can prove this equality in the same way as in \cite{Ito}.
\end{proof}

Let $\alpha >0$, and define a mapping $T_t$ from Borel measurable functions on ${\bf{R}}^d$ to themselves by \( T_t f(x):= 2^{-\alpha t }f(2^t x)\) for $t\in {\bf R}$.
Then, by a similar argument to the proof of Lemma \ref{lm2-2} we have the following corollary.

\begin{co}\label{co2-2}
Suppose that the law of $T_t W$ equals to that of $W$ for all $t$ and that
\[
\lim _{t\rightarrow \infty} 2^{-\alpha t}\sup _{x,y\in S^{d-1}} K(2^t x,y)=0,
\]
then the family of the transformations $\{ T_t\}$ is strongly mixing.
\end{co}

By combining the results in Section \ref{General} and the results above, we obtain the following criterion for the recurrence of the Brownian motion in semi-selfsimilar Gaussian environments. 

\begin{Th}\label{th-gauss}
Let $W$ be a Gaussian field on ${\bf{R}}^d$ such that $W$ is continuous on ${\bf{R}}^d$ almost surely, $W(0)=0$ almost surely, and that $E[W(x)]=0$ for $x\in {\bf{R}}^d$, and let $K(x,y):=E[W(x)W(y)]$ for $x,y \in {\bf R}^d$. 
Assume that
\begin{enumerate}
\item \label{th-gauss-1}there exists a positive constant $\varepsilon$ such that
\[
\inf _{x\in D_1}\int _{D_1} K(x,y)dy \geq \varepsilon ,
\]
\item \label{th-gauss-2}the law of $T_n W$ equals to that of $W$ for all $n \in {\bf Z}$ and that
\[
\lim _{n\rightarrow \infty} r^{-\alpha n}\sup _{x,y\in D_1} K(r^nx,y)=0.
\]
\end{enumerate}
Then, the diffusion process $X_W$ associated with the generator (\ref{generator}) is recurrent for almost all environments $W$.
\end{Th}

\begin{proof}
In view of Lemma \ref{lm2-1}, the assumption \ref{th-gauss-1} implies that there exists a positive constant $a$ satisfying that
\[
P\left( \inf _{x\in D_1}W(x)\geq a \right) >0.
\]
On the other hand, in view of Lemma \ref{lm2-2}, the assumption \ref{th-gauss-2} implies that $T_n$ is a strong mixing transformation for $P\circ W^{-1}$.
Therefore, by Theorem \ref{Th1-3} we have the assertion.
\end{proof}

By applying Corollaries \ref{co2-1}, \ref{co2-2}, and Theorem \ref{Th1-3}, instead of Lemmas \ref{lm2-1}, \ref{lm2-2}, and Theorem \ref{Th1-2} in the proof of Theorem \ref{th-gauss}, we have the following theorem for selfsimilar Gaussian environments.

\begin{Th}\label{th-gauss2}
Let $W$ be a Gaussian field on ${\bf{R}}^d$ such that $W$ is continuous on ${\bf{R}}^d$ almost surely, $W(0)=0$ almost surely, and that $E[W(x)]=0$ for $x\in {\bf{R}}^d$, and let $K(x,y):=E[W(x)W(y)]$ for $x,y\in {\bf R}^d$. 
Assume that
\begin{enumerate}
\item \label{th-gauss2-1}there exists a positive constant $\varepsilon$ such that
\[
\inf _{x\in S^{d-1}}\int _{S^{d-1}} K(x,y)dy \geq \varepsilon ,
\]
\item \label{th-gauss2-2}the law of $T_tW$ equals to that of $W$ for all $t\in {\mathbf R}$ and that
\[
\lim _{t\rightarrow \infty} 2^{-\alpha t}\sup _{x,y\in S^{d-1}} K(2^t x,y)=0.
\]
\end{enumerate}
Then, the diffusion process $X_W$ associated with the generator (\ref{generator}) is recurrent for almost all environments $W$.
\end{Th}

\section{The Brownian motion on fractional Brownian fields}\label{sec:fBf}

In this section, we consider the case that environments are fractional Brownian fields.
We can apply Theorem \ref{th-gauss2} to this case, and show the recurrence of the Brownian motion on fractional Brownian fields.

For given $H\in (0,1)$, let $W$ be a Gaussian random environment which satisfying that $W(0)=0$ almost surely, \( E[W(x)]=0\) for $x\in {\bf{R}}^d$, and the covariance between $W(x)$ and $W(y)$
\[
K(x,y):= \frac12 \left( |x|^{2H}+|y|^{2H}-|x-y|^{2H} \right) , \quad x,y \in {\bf{R}}^d.
\]
Note that the law of Gaussian random environments is determined by the mean and the covariance.
The random field $W$ is called a fractional Brownian field. When $H=1/2$, it is called L\'evy's Brownian motion (c.f. \cite{Tanaka}).
It is easy to see that the environment $W$ is a selfsimilar random environment with $\alpha =H$ in the sense in Section \ref{General}.
The parameter $H$ is called the Hurst parameter.

Now we show the following theorem as an application of Theorem \ref{th-gauss2}.

\begin{Th}\rm
Let $W$ be a fractional Brownian field with the Hurst parameter $H\in (0,1)$.
Then, the process $X_W$ given by the generator (\ref{generator}) is recurrent for almost all environments $W$.
\end{Th}

\begin{proof}
It is sufficient to check the assumptions \ref{th-gauss2-1} and \ref{th-gauss2-2} of Theorem \ref{th-gauss2}.

First we check the assumption \ref{th-gauss2-1} of Theorem \ref{th-gauss2}.
When $d=1$, it is easy to see that the assumption \ref{th-gauss2-1} of Theorem \ref{th-gauss2} holds.
Let $d\geq 2$.
From the rotation invariance of $\int _{S^{d-1}} K(x,y) dy$ in $x$, we have
\begin{equation}
\inf _{x\in S^{d-1}}\int _{S^{d-1}} K(x,y)dy= c_d \int _{-\pi}^{\pi} K(e_1,v_\theta )d\theta , \label{ex2-1-1}
\end{equation}
where $e_1=(1,0,0, \dots ,0) \in {\bf{R}}^d$, $v_\theta =(\cos \theta ,\sin \theta ,0, \dots ,0) \in {\bf{R}}^d$, and $c_d=1$ for $d=2$ and $c_d$ is the surface area of $S^{d-2}$ for $d\geq 3$. When \( 0<H\leq{\frac 12}\),
\[
K(e_1 ,v_\theta )= \frac 12 \left( 2-|e_1-v_\theta |^{2H} \right) \geq 0, \quad \theta \in (-\pi ,\pi ].
\]
Hence, the assumption \ref{th-gauss2-1} of Theorem \ref{th-gauss2} holds.
We consider the case that \( {\frac 12}<H<1\).
Define \( \theta _0 \in [0,\pi ]\) by $\cos \theta _0 = 1- 2^{{\frac 1H} -1}.$
Then, \( {\frac \pi 2} <\theta _0 < \pi \), $2-|e_1 -v_\theta |^{2H} >0$ for $\theta \in [0, \theta _0)$ and $2-|e_1 -v_\theta |^{2H} <0$ for $\theta \in (\theta _0 ,\pi ]$.
Hence,
\begin{eqnarray*}
&& \int _{-\pi}^{\pi} K(e_1,v_\theta )d\theta\\
&&= \int _0^\pi \left\{ 2-|e_1 -v_\theta |^{2H} \right\} d\theta \\
&&= \int _0^{\pi -\theta _0} \left\{ 2 - (2-2\cos \theta )^H\right\} d\theta + \int _{\pi -\theta _0} ^{\theta _0} \left( 2-|e_1 -v_\theta |^{2H} \right) d\theta + \int _{\theta _0}^{\pi } \left\{ 2 - (2-2\cos \theta )^H\right\} d\theta \\
&&= 2^{H}\int _0^{\pi -\theta _0} \left\{ 2^{1-H} - (1-\cos \theta )^H\right\} d\theta + \int _{\pi -\theta _0}^{\theta _0} \left( 2-|e_1 -v_\theta |^{2H} \right) d\theta \\
&&\hspace{9cm} + 2^{H}\int _{\theta _0}^\pi \left\{ 2^{1-H} - (1-\cos \theta )^H\right\} d\theta\\
&&= 2^{H}\int _0^{\pi -\theta _0} \left\{ 2^{2-H} - (1-\cos \theta )^H - (1+\cos \theta )^H\right\} d\theta + \int _{\pi -\theta _0}^{\theta _0} \left( 2-|e_1 -v_\theta |^{2H} \right) d\theta .
\end{eqnarray*}
Since \( (1-\cos \theta )^H + (1+\cos \theta )^H \leq 2\) for \( \theta \in [0,\pi ]\), we have
\[
\int _{-\pi}^{\pi} K(e_1,v_\theta )d\theta \geq \int _{\pi -\theta _0}^{\theta _0} \left( 2-|e_1 -v_\theta |^{2H} \right) d\theta > 0.
\]
Thus, from (\ref{ex2-1-1}) we see that the assumption \ref{th-gauss2-1} of Theorem \ref{th-gauss2} holds.

Next, we check the assumption \ref{th-gauss2-2} of Theorem \ref{th-gauss2}.
Let $x,y\in S^{d-1}$ and $t \in {\mathbf R}$.
Then, we have
\begin{eqnarray*}
2^{-Ht} K(2^t x,y)
&=& 2^{-Ht-1} \left( |2^t x|^{2H}+|y|^{2H}-|2^t x-y|^{2H} \right) \\
&=& 2^{Ht-1} \left( 1+2^{-2Ht}-|x-2^{-t}y|^{2H} \right).
\end{eqnarray*}
Let \( f_{xy}(s)=|x-sy|^{2H}\) for \( s\in [0,1]\).
In view of the mean-value theorem, we have that there exists \( \theta _t \in [0,2^{-t}]\) such that 
\[
\partial _s f_{xy}(\theta _t ) = 2^t \left( 1-|x-2^{-t}y|^{2H} \right) .
\]
On the other hand, it is easy to see that
\[
\limsup _{t\rightarrow \infty} \sup _{x,y\in S^{d-1}} |\partial _s f_{xy}(\theta _t )|<\infty .
\]
Hence,
\begin{eqnarray*}
\limsup _{t\rightarrow \infty} 2^{-Ht}\sup _{x,y\in S^{d-1}} \left| K(2^t x,y) \right|
&=& \limsup _{t\rightarrow \infty} 2^{Ht-1} \sup _{x,y\in S^{d-1}} \left| 2^{-2Ht} - 2^{-t} (\partial _s f_{xy})(\theta _t) \right| \\
&\leq& \limsup _{t\rightarrow \infty} \left( 2^{-Ht-1} + 2^{(H-1)t-1}\sup _{x,y\in S^{d-1}} \left| \partial _s f_{xy}(\theta _t ) \right| \right) \\
&=& 0,
\end{eqnarray*}
which implies that the assumption \ref{th-gauss2-2} of Theorem \ref{th-gauss2} holds.
\end{proof}

\section*{Acknowledgment}
This work is supported by JSPS KAKENHI Grant number 25800054 and 26800063.

\end{document}